\documentclass[12pt]{amsart}
\usepackage{amsmath,amscd, nicefrac, xcolor, setspace, amsfonts, amssymb, bbm, amsthm, tikz, tikz-cd, enumerate, graphicx,verbatim,todonotes,mathtools,url}
\usetikzlibrary{calc,decorations.markings}
\usepackage[margin=1in]{geometry}
\usepackage[shortalphabetic]{amsrefs}

\newtheorem{thm}{Theorem}[section]
\newtheorem{theorem}[thm]{Theorem}

\newtheorem{proposition}[thm]{Proposition}
\newtheorem{lemma}[thm]{Lemma}
\newtheorem*{lemma*}{Lemma}

\newtheorem{definition}[thm]{Definition}

\newtheorem{remark}[thm]{Remark}

\newcommand{\Q}{\mathbb{Q}}
\newcommand{\Z}{\mathbb{Z}}

\newcommand{\FF}{\mathbb{F}_2}

\newcommand{\Gal}{\mathrm{Gal}}

\newcommand{\Sel}{\mathrm{Sel}_2}

\newcommand{\dimf}{\dim_{\mathbb{F}_2}}

\newcommand{\rk}{\mathrm{rk}}
\newcommand{\nrm}{\mathrm{\mathbf{N}}}

\usepackage[T2A,T1]{fontenc}
\DeclareSymbolFont{cyrillic}{T2A}{cmr}{m}{n}
\DeclareMathSymbol{\Sha}{\mathalpha}{cyrillic}{216}

\numberwithin{equation}{section}

\allowdisplaybreaks

\usepackage[pagewise]{lineno}

\title[Rank growth of elliptic Curves in $S_4$ and $A_4$-quartic extensions]{Rank growth of Elliptic Curves in $S_4$ and $A_4$-quartic extensions of the rationals} 
\date{\today}
\author{Daniel Keliher}
\address{Massachusetts Institute of Technology, Cambridge, MA 02139, USA}
\email{keliher@mit.edu}
\begin{document}

\begin{abstract}
We investigate the rank growth of elliptic curves from $\mathbb{Q}$ to $S_4$ and $A_4$-quartic extensions $K/\mathbb{Q}$. In particular,  we are interested in the quantity $\mathrm{rk}(E/K) - \mathrm{rk}(E/\mathbb{Q})$ for fixed $E$ and varying $K$. When $\mathrm{rk}(E/\mathbb{Q}) \leq 1$, with $E$ subject to some other conditions, we prove there are infinitely many $S_4$-quartic extensions $K/\mathbb{Q}$ over which $E$ does not gain rank, i.e. such that $\mathrm{rk}(E/K) - \mathrm{rk}(E/\mathbb{Q}) = 0$. To do so, we show how to control the 2-Selmer rank of $E$ in certain quadratic extensions, which in turn contributes to controlling the rank in families of $S_4$ and $A_4$-quartic extensions of $\mathbb{Q}$.
\end{abstract}
\maketitle
\section{Introduction}

\subsection{Rank Growth}
For a number field $L$ and an elliptic curve $E$ defined over $L$, let $E(L)$ be the group of $L$-rational points of $E$. The Mordell-Weil Theorem says $E(L)$ is a finitely generated abelian group. The rank of $E(L)$, denoted $\mathrm{rk}(E/L)$, has been the subject of much study. Of particular interest here is the behavior of the rank upon base change, i.e. for an extension of number fields $K/L$, what is $\mathrm{rk}(E/K) - \mathrm{rk}(E/L)$? We call this difference the \textit{rank growth} of $E$ in $K/L$.

Suppose $L=\mathbb{Q}$ and $K/\mathbb{Q}$ denotes a quadratic extension. Given an elliptic curve $E/\mathbb{Q}$, a conjecture of Goldfeld predicts that 50\% of quadratic twists, $E^K$, of $E$ have analytic rank zero and 50\% have analytic rank one. See Section \ref{subsec:quadtwists} for a definition and discussion of quadratic twists. Recent work of Smith \cites{Smith0, Smith1, Smith2} studies the distribution of $\ell^\infty$-Selmer groups and proves a version of Goldfeld's Conjecture for $2^\infty$-Selmer coranks. 

In this paper, we will be interested in studying rank growth in higher degree and non-abelian extensions. In this setting, ranks of quadratic twists, $E^K$, measure the rank growth of $E$ from $\mathbb{Q}$ to $K$, which will be essential for rank growth in some larger degree extensions.  Previously, and in higher degrees, David, Fearnley, and Kisilevsky \cite{David} have given conjectures for how frequently the rank of an elliptic curve grows in cyclic prime degree extensions. Lemke Oliver and Thorne \cite{LOT} gave asymptotic lower bounds for the number of $S_d$ extensions for which an elliptic curve $E$ gains rank.  Further, Shnidman and Weiss \cite{Ari} study rank growth of elliptic curves from a number field $L$ up to an extension $L(\sqrt[2n]{d})$.  

\subsection{Results for $S_4$ and $A_4$-quartic extensions}
In this paper, we investigate the rank growth of elliptic curves in $S_4$ and $A_4$-quartic extensions of the rationals. In what follows, unless stated otherwise, we will always assume $E$ is an elliptic curve defined over $\mathbb{Q}$. Further,  $K/\mathbb{Q}$ will always be an $S_4$ or $A_4$-quartic extension. That is, one for which the normal closure of $K$ over $\mathbb{Q}$ is an $S_4$ or $A_4$-Galois extension. For an elliptic curve $E/\mathbb{Q}$ with discriminant $\Delta_E$, we'll consider $\mathrm{rk}(E/K) - \mathrm{rk}(E/\mathbb{Q})$ for many such $K$. It will often be convenient to use the rank of a Selmer group of an elliptic curve $E/K$ (here we only need the 2-Selmer group) in place of the rank, $\mathrm{rk}(E/K)$. See Definition \ref{def:Sel2} for the definition of the 2-Selmer group, $\mathrm{Sel}_2(E/K)$, of an elliptic curve $E/K$ and Section \ref{seldefs} for a discussion of their utility in our context. 

In particular, for an elliptic curve $E$ with Selmer rank zero over the rationals, subject to some mild constraints, we prove there are infinitely many $S_4$ (or $A_4$) quartic extensions over which $E$ does not gain rank. Further, we give a lower bound on the number of such extensions with bounded discriminant with some fixed cubic resolvent field.

\begin{theorem}\label{thm:mainrank}
Let $E$ be an elliptic curve defined over $\Q$ such that $\Gal(\mathbb{Q}(E[2])/\Q) \simeq S_3$,
 and $\Sel(E/\Q)=0$.
Let $K_3$ be an $S_3$-cubic (resp., $C_3$-cubic) extension of $\Q$ such that 
$\tilde K_3$ and $\mathbb{Q}(E[2])$ are linearly disjoint.
Suppose also that there is a place $v_0$ of $K_3$, 
unramified in $\tilde K_3$, such that either $v_0$ is real and 
$\Delta_E <0$, or $v_0 \nmid 2\infty$, $E$ has multiplicative reduction at $v_0$ and $\mathrm{ord}_{v_0}(\Delta_E)$ is odd. Then there are infinitely many $S_4$-quartic (resp., $A_4$-quartic) extensions $K$ over $\Q$ with cubic resolvent $K_3$ such that
\begin{itemize}
    \item if $\dim_{\FF}\Sel(E/K_3) \equiv 0 \pmod{2}$, then $\rk(E/K)=0$;
    \item if $\dim_{\FF}\Sel(E/K_3) \equiv 1 \pmod{2}$, and also assuming the parity condition \\ $\dim_{\FF}\Sel(E/K_3) \equiv \rk(E/K_3) \pmod{2}$, then $\rk(E/K)=1$.
\end{itemize}
\end{theorem}

To prove the theorem above, we utilize two tools. First, we reduce the problem of studying rank growth in our quartic extensions, $K/\mathbb{Q}$, to one of studying rank growth in certain quadratic subextensions of the Galois closure of $K$, and thus to studying the rank of certain quadratic twists. Second, we use and further develop some Selmer group machinery of Mazur and Rubin. Indeed, in \cite{MR}, they show that under suitable assumptions, an elliptic curve has infinitely many twists of a prescribed Selmer rank. In subsequent work, Klasgbrun, Mazur, and Rubin \cite{KMR} study the distribution of 2-Selmer ranks of quadratic twists of an elliptic curve. The quadratic twists we study here, see Definition \ref{def:sqtwist} below, are thin in the full family of quadratic twists, so we require a different approach. Nonetheless, we use similar ideas to show one can reduce the 2-Selmer ranks of the appropriate quadratic twists either to one or zero. Translating from the language of Selmer groups to ranks yields Theorem \ref{thm:mainrank}.

To emphasize the extra properties imposed on the quadratic twists we consider here, we make the following definition. 

\begin{definition}[Square Norm Twists]\label{def:sqtwist}
Let $E$ be an elliptic curve defined over a number field $L$, and let $F/L$ be a quadratic extension. We call a quadratic twist $E^F$ over $L$ a \emph{square norm twist} if $F=L(\sqrt{\alpha})$ where $\alpha \in L^\times/(L^\times)^2$ and $\nrm_{L/\mathbb{Q}}(\alpha)$ is a square. 
\end{definition}

Such twists will be key to keeping track of an associated $S_4$-quartic extension when working over a suitable cubic field (see Lemma \ref{lem:quadforS4}).

Likewise, to streamline the discussion and highlight the properties required of the cubic resolvents of our $S_4$ (or $A_4$) quartic extensions, which are $S_3$ (or $C_3$) cubic extensions of the rationals, always denoted $K_3$, we make the next definition.
\begin{definition}[Admissible Cubic Resolvent]\label{def:admissibleK3}
Let $K_3$ be a cubic extension of the rationals and fix an elliptic curve 
$E/\mathbb{Q}$ with discriminant $\Delta_E$. Suppose $E(K_3)[2]=0$, 
$\tilde K_3$ (the Galois closure of $K_3$ over $\mathbb{Q}$) and $\mathbb{Q}(E[2])$ are linearly disjoint, and there is a place $v_0$ of $K_3$, 

unramified in $\tilde K_3$, such that either $v_0$ is real and $(\Delta_E)_{v_0} < 0$ or $v_0 \nmid 2\infty$, $E$ has multiplicative reduction at $v_0$ and $\mathrm{ord}_{v_0}(\Delta_E)$ is odd. For such an extension,

\begin{itemize}
    \item if $K_3$ is an $S_3$-cubic extensions, 
    we call $K_3$ an \emph{admissible $S_3$-cubic resolvent for $E$};
    \item if $K_3$ is a $C_3$-cubic extension, 
    we call $K_3$ an \emph{admissible $C_3$-cubic resolvent for $E$}.
\end{itemize}
\end{definition}

In the event we do not need to specify one of the two Galois group cases above, we will call a $K_3$ as in one of the two cases above an \emph{admissible cubic resolvent} for some $E$.

Note that the restrictions placed on $K_3$ to make it admissible for some elliptic curve are not overly burdensome.
The conditions on the distinguished place should be compared to the assumptions of \cite{MR}*{Theorem 1.6}.

Theorem \ref{thm:mainrank} is a consequence of the following result.

\begin{theorem}\label{thm:mainsel}
Let $E$ be an elliptic curve defined over $\Q$ such that $\Gal(\mathbb{Q}(E[2])/\Q) \simeq S_3$,
 and $\Sel(E/\Q)=0$, and 
let $K_3$ be an admissible cubic resolvent for $E$.
Then there are infinitely many square norm twists, $E^F/K_3$ such that
\begin{itemize}
    \item If $\dim_{\FF}\Sel(E/K_3) \equiv 0 \pmod{2}$, then $\dim_{\FF} \Sel(E^F/K_3) = 0$ ;
    \item if $\dim_{\FF}\Sel(E/K_3) \equiv 1 \pmod{2}$, then $\dim_{\FF} \Sel(E^F/K_3) = 1$.
\end{itemize}
\end{theorem}

\begin{remark}
In Theorems \ref{thm:mainrank} and \ref{thm:mainsel}, and in what follows, when we say ``infinitely many'' we mean the number of such things with the norm of their discriminant bounded above by $X$ is $\gg X^{1/2}/\log{X}^\alpha$ for some $\alpha > 0$. Further, $\alpha$ depends on $E$ and the two torsion field $K_3(E[2])$. The sizes of these lower bounds, particularly the values of $\alpha$ depending on $K_3$ and $E$, are elucidated in Proposition \ref{prop:manytwists}.
\end{remark}

\begin{remark}
In the case that one obtains infinitely many $A_4$-quartic extensions for which the Mordell or 2-Selmer rank has some prescribed behavior, the ``infinitely many'' given by the two theorems above differs from the predicted number of $A_4$-quartic extensions only by some factors of $\log X$. In particular, Malle's conjecture \cite{Malle} predicts that the number of $A_4$-quartic extensions of a number field $L$ with absolute value of the norm of the relative discriminant bounded by $X$ is asymptotic to $c_LX^{1/2}\log{X}^{b_L}$ for constants $b_L$ and $c_L$ depending on $L$.  

The number of $S_4$-quartic extensions of $L$ with absolute value of the norm of the relative discriminant bounded by $X$ is asymptotic to $d_LX$ for a constant $d_L$ depending on $L$ \cite{ BSW}. In this case the ``infinitely many'' given by the two theorems above differs from $S_4$-quartic asymptotic by both a logarithmic factor and a factor of $X^{1/2}$.
\end{remark}

\subsection{Layout}\label{sec:layout}In Section 2 we outline the connection between rank growth in $S_4$ and $A_4$-quartic extensions with rank growth in certain quadratic extensions. In Section 3, we recall some facts about Selmer groups and record some related results of Mazur and Rubin \cite{MR} on quadratic twists. In Sections 4 and 5, we interface the tools of Section 3 with the notion of square norm twists to show we can decrease the 2-Selmer rank of an elliptic curve with a suitable square norm twist and can indeed find many such twists. Finally, in Section 6, we prove the main theorems stated above.

\section{Rank Growth in $S_4$-quartics}\label{sec:rankgrowth} 

\subsection{Preliminaries}

Note, for an extension of number fields $L/\mathbb{Q}$, we write $\tilde L$ for the Galois closure of $L$ in some choice of algebraic closure $\bar{\mathbb{Q}}$.

Consider an $S_4$ (or $A_4$) extension $K/\Q$ with Galois closure $\tilde K$. We are principally concerned with the change (or lack of change) in rank in the group of $K$-rational points vs. the group of $\mathbb{Q}$-rational points of $E$. We will show this rank change is governed by the rank growth in a quadratic extension of fields between $\Q$ and $\tilde K$.  $K_3$ will always denote a cubic resolvent for our quartic extension(s). Of particular interest will be fixing an admissible cubic resolvent $K_3$ and considering many quartic $S_4$ (or $A_4$) extensions $K$ with cubic resolvent $K_3$.

\subsection{$S_4$ and $A_4$-quartic extensions}
Before we turn to the question of rank growth, we record a few facts about $S_4$ and $A_4$-quartic extensions with cubic resolvent field $K_3$.

\begin{lemma}\label{lem:genericwr}
Let  $K_3/\mathbb{Q}$ be a cubic extension and $F$ be a quadratic extension of $K_3$. There is always an embedding $\mathrm{Gal}(\tilde F/\mathbb{Q})\hookrightarrow S_2 \wr S_3.$
\end{lemma}
Quadratic extensions of $S_3$-cubics generically have Galois group $S_2\wr S_3$ over $\mathbb{Q}$. We are interested in the case where the Galois group is instead $S_4 < S_2 \wr S_3$. Likewise for $C_3$-cubics, we wish to consider the case where quadratic extensions of $K_3$ have Galois group $A_4$.

\begin{lemma}\label{lem:quadforS4}
Fix an $S_3$ or $C_3$-cubic extension $K_3/\mathbb{Q}$, and let $F$ denote a quadratic extension of the form $K_3(\sqrt{\alpha})$, where $\alpha \in K_3$ and $\nrm_{K_3/\mathbb{Q}}(\alpha)$ is a square.

\begin{itemize}
    \item If $K_3/\mathbb{Q}$ is an $S_3$-cubic, then $\Gal(\tilde F/\mathbb{Q}) \simeq S_4$. Further, there is a one-to-one correspondence between such quadratic extensions $F/K_3$ and $S_4$-quartic extensions of $\mathbb{Q}$ with cubic resolvent $K_3$.
    \item If $K_3/\mathbb{Q}$ is a $C_3$-cubic, then $\Gal(\tilde F/\mathbb{Q}) \simeq A_4$. Further, there is a three-to-one correspondence between such quadratic extensions $F/K_3$ and $A_4$-quartic extensions of $\mathbb{Q}$ with cubic resolvent $K_3$.
\end{itemize}
\end{lemma} 
This correspondence is described in detail in, for example, Section 2 of \cite{CohenThorneQuartic}.

\begin{remark}
Fix a cubic field, $K_3$. For each $L=K_3(\sqrt{\alpha})$ where $\alpha \in K_3^\times/(K_3^\times)^2$  and $N_{K_3/\mathbb{Q}}(\alpha)$ is a square, there is an $S_4$ or $A_4$-quartic extension $K/\mathbb{Q}$ with cubic resolvent $K_3$. Use Lemma \ref{lem:quadforS4} and observe that $\tilde L = \tilde K$. For our purposes, we will fix $K_3$ and range over quadratic extensions of $K_3$ as in Lemma \ref{lem:quadforS4} to range over $S_4$-quartic extensions of $\mathbb{Q}$ with cubic resolvent $K_3$. We will then consider a fixed elliptic curve $E$ over these extensions, and consider various differences in rank. 
\end{remark}

\subsection{Measuring rank growth in $S_4$ and $A_4$-quartic extensions}

Our aim now is to show that measuring rank growth of an elliptic curve $E$ from $\mathbb{Q}$ to $S_4$ or $A_4$-quartic extensions $K/\mathbb{Q}$ with cubic resolvent $K_3/\mathbb{Q}$ is a matter of measuring the rank growth of $E$ from $K_3$ to a quadratic extension $F/K_3$, namely the quadratic extension of Lemma \ref{lem:quadforS4}. 

\begin{lemma}\label{lem:rankgrowth}
Let $E$ be an elliptic curve defined over $\Q$, $K$ be an $S_4$ or $A_4$-quartic extension of $\mathbb{Q}$ with cubic resolvent $K_3$, and $F/K_3$ be the quadratic extension of Lemma \ref{lem:quadforS4}. Then

\begin{equation}
    \mathrm{rk}(E/K) - \mathrm{rk}(E/\mathbb{Q}) = \mathrm{rk}(E/F) - \mathrm{rk}(E/K_3).
\end{equation}
\end{lemma}

The rank relation of Lemma \ref{lem:rankgrowth} is a manifestation of the following more general fact \cite{DD}*{Page 572}. Suppose $L/k$ is a Galois extension of number fields with $G = \mathrm{Gal}(L/k)$, and $E/k$ is an elliptic curve. For $H \leq G$, write $\mathbf{1}_H$ for the trivial character on $H$. If there are subextensions $K_i/k$ and $K_j'/k$ of $L$,  cut out by subgroups $H_i$ and $H'_j$ of $G$, such that
\begin{equation}\label{eq:brauerrel}
\bigoplus_i \mathrm{Ind}_{H_i}^G \mathbf{1}_{H_i} \simeq  \bigoplus_j \mathrm{Ind}_{H_j'}^G \mathbf{1}_{H_j'}
\end{equation}
as complex representations of $G$, then
\begin{equation}\label{eq:rankrel}
\sum_i \mathrm{rk}(E/K_j) = \sum_j \mathrm{rk}(E/K'_j).
\end{equation}
The see this relation on the ranks, let $\chi_L$ be the character of the representation of the complex representation of the Mordell-Weil group, $E(L) \otimes \mathbb{C}$, and note that 
\begin{equation}\label{eq:rankprod}
\mathrm{rk}(E/K_i) = \langle  \mathrm{Ind}_{H_j}^G  \mathbf{1}_{H_i}, \chi_L \rangle.
\end{equation}
The same statement can be made for the rank of $E$ over each $K_j'$ as well. Using \eqref{eq:rankprod} together with \eqref{eq:brauerrel} yields \eqref{eq:rankrel}.
\begin{proof}[Proof of Lemma \ref{lem:rankgrowth}]
Consider the following subgroups of $G = \mathrm{Gal}(\tilde K/\mathbb{Q})\simeq S_4$ which fix the subfields $K,K_3$, and $F$ of $\tilde K$. Let $H_K \simeq S_3$ be the subgroup fixing $K$, $H_{K_3} \simeq D_8$  be the subgroup fixing $K_3$, and $H_F \simeq V_4$, the Klein four group, be the subgroup fixing $F$. Then one can verify 
\begin{equation}\label{eq:S4brauer}
\mathrm{Ind}_{H_{K}}^G \mathbf{1}_{H_K} \oplus \mathrm{Ind}_{H_{K_3}}^G \mathbf{1}_{H_{K_3}}  \simeq \mathbf{1} \oplus \mathrm{Ind}_{H_F}^G \mathbf{1}_{H_F}.
\end{equation}
The lemma now follows from \eqref{eq:rankrel}. Note that relations like that of \eqref{eq:S4brauer} is an example of those provided in \cite{BD}.
\end{proof}

\section{The 2-Selmer Groups and Quadratic Twists}\label{sec:twistsandsel}
In the previous section, we established that the rank growth from $\Q$ to $K$ is the same as the rank growth from $K_3$ to a quadratic extension of $K_3$ determined by $K$. So we may restrict ourselves to the study of rank growth in quadratic extensions. This is governed by the theory of quadratic twists. 

\subsection{The 2-Selmer Group} \label{seldefs}
We now recall the definition of the 2-Selmer group for an elliptic curve $E$ over a number field $L$. The multiplication-by-2 map on $E$ gives rise to a short exact sequence of Galois modules: 
$$0 \rightarrow E[2] \rightarrow E(\bar{\mathbb{Q}}) \overset{\times 2}{\rightarrow}E(\bar{\mathbb{Q}}) \rightarrow 0.$$

This in turn yields a long exact sequence of Galois cohomology groups, which, after quotienting appropriately, gives rise to the following diagram.

\[
\begin{tikzcd}[column sep=small]
0 \arrow[r] & E(L)/2E(L) \arrow[r] \arrow[d]   & {H^1(L, E[2])} \arrow[r] \arrow[d] & {H^1(L,E)}[2] \arrow[d] \arrow[r] & 0 \\
0 \arrow[r] & \prod_v E(L_v)/2E(L_v) \arrow[r] & {\prod_v H^1(L_v, E[2])} \arrow[r] & {\prod_v H^1(L_v, E)}[2] \arrow[r] & 0
\end{tikzcd}
\]
Now, define subgroups  $H^1_f(L_v, E[2])$ of each local cohomology $H^1(L_v, E[2])$ as 
$$H^1_f(L_v, E[2]) \coloneqq \mathrm{Image}\left( E(L_v)/2E(K_v) \rightarrow H^1(L_v, E[2]) \right).$$

\begin{definition}\label{def:Sel2}
The \textit{2-Selmer group of $E/L$}, denoted $\Sel(E/L)$, is the $\FF$-vector space defined by the exactness of the following sequence:
$$0 \rightarrow \Sel(E/L) \rightarrow H^1(L, E[2]) \rightarrow \bigoplus_{v}H^1(L_v, E[2])/H^1_f(L_v, E[2]).$$
\end{definition}
We may think of the elements of the 2-Selmer group as being the classes in $H^1(L, E[2])$ which, for every place $v$ of $L$, land in the image of $E(L_v)/2E(L_v)$. That is, elements of $H^1(L, E[2])$ which everywhere locally satisfy the local conditions determined by  $H^1_f(L_v, E[2])$.

Further, the 2-Selmer group fits into a short exact sequence
$$0 \rightarrow E(L)/2E(L) \rightarrow \Sel(E/L) \rightarrow \Sha(E/L)[2] \rightarrow 0,$$
where $\Sha(E/L)[2]$ are the elements of the Shafarevich-Tate group of $E/L$ with order dividing 2. We have
\begin{equation}\label{eq:rkbound}
    \dim_{\FF}\Sel(E/L) = \left(\mathrm{rk}(E/L) + \dim_{\FF}E(L)[2]\right) + \dim_{\FF}\Sha(E/L)[2]
\end{equation}
and further that $\mathrm{rk}(E/L) \leq \dim_{\FF}\Sel(E/L).$ It is by this relation that we'll access the ranks of the various elliptic curves and twists discussed later in the paper. 

\subsection{Quadratic Twists}\label{subsec:quadtwists}
Suppose our elliptic curve $E/L$ is given in short Weierstrass form
$$E: y^2=x^3+Ax+B$$
with $A,B \in L$. A \textit{quadratic twist}, $E^F/L$, of $E/L$ is an elliptic curve of the form
\begin{equation}\label{eq:twisted}
E^F: \delta y^2 = x^3+Ax+B
\end{equation}
where $\delta \in L^\times/ (L^\times)^2$ and $F = L(\sqrt{\delta})$. With a change of variables, one can put \eqref{eq:twisted} in short Weierstrass form:
$$E^F : y^3 = x^3+a\delta^2x + b\delta^3.$$

An elliptic curve $E/L$ and a quadratic twist $E^F/L$ are not, in general, isomorphic as elliptic curves over $L$ but \textit{are} isomorphic as elliptic curves over $F$. In particular, quadratic twists will be the main tool for measuring growth in quadratic extensions as we have
$$\mathrm{rk}(E^F/L) = \mathrm{rk}(E/F)-\mathrm{rk}(E/L).$$

\subsection{The 2-Selmer of quadratic twists}
In \cite{MR}, Mazur and Rubin give results in which understanding the behavior of an elliptic curve $E/L$ and its 2-Selmer group, $\text{Sel}_2(E/L)$, locally at only a few places of $L$ is sufficient to, under some mild conditions, understand the relation between $\dim_{\mathbb{F}_2}\text{Sel}_2(E/L)$ and $\dim_{\mathbb{F}_2}\text{Sel}_2(E^F/L)$ for some quadratic twists, $E^F$, of $E$.

One defines the 2-Selmer group of a twist $E^F/L$ following the definition laid out in Section \ref{seldefs}, just with $E^F$ in place of $E$. Important to note that is $E^F[2]$ and $E[2]$ are isomorphic as Galois modules, so we may view both Selmer groups inside $H^1(L, E[2])$.

Lemma \ref{lem:rankgrowth} shows that understanding the rank growth in quadratic extensions will be sufficient for understanding rank growth in the quartic extensions of interest. In the reminder of this section we will record some results from \cite{MR} on how local information about $E$ relates the 2-Selmer rank of $E$ to the 2-Selmer rank of quadratic twists of $E$. 

\begin{lemma}[\cite{MR}*{Lemma 2.2}]\label{lem:h1f}
With the notation as above:
\begin{itemize} 
    \item If $v \nmid 2\infty$, then $\dimf H^1_f(L_v, E[2]) = \dimf E(L_v)[2]$.
    \item If $v \nmid 2\infty$ and $E$ has good reduction at $v$, then 
    $$H^1_f(L_v, E[2]) \cong E[2]/(\mathrm{Frob}_{v} -1)E[2].$$
\end{itemize}
\end{lemma}

\begin{definition}
If $T$ is a finite set of places of $L$. Let $\mathrm{loc}_T$ be the sum of the localization maps for each place of $T$, 
$$\mathrm{loc}_T: H^1(L, E[2]) \rightarrow \bigoplus_{v \in T} H^1(L_v, E[2]).$$
Also set 
$$V_T = \mathrm{loc}_T\left( \Sel(E/L)\right) \subset \bigoplus_{v \in T} H^1_f(L_v, E[2]).$$
\end{definition}

We finish the section be recalling two results from \cite{MR} that we'll later use to control the rank of the $2$-Selmer groups in the quadratic extension of Lemma \ref{lem:rankgrowth}.

\begin{lemma}[\cite{MR}*{Proposition 3.3}]\label{MR3.3}
Let $E/L$ be an elliptic curve, and let $F/L$ be a quadratic extension in which the following places of $L$ split:
\begin{itemize}
    \item all primes where $E$ has additive reduction;
    \item all places $v$ where $E$ has multiplicative reduction such that $\mathrm{ord}_v(\Delta_E)$ is even;
    \item all primes above 2;
    \item all real places $v$ with $(\Delta_E)_v > 0$.
\end{itemize}
Further, suppose that all $v$ where $E$ has multiplicative reduction and $\mathrm{ord}_v(\Delta_E)$ is odd are unramified in $F/L$. 

Let $T$ be the set of finite primes $\mathfrak p$ of $L$ such that $F/L$ is ramified at $\mathfrak p$ and $E(L_\mathfrak{p})[2]\neq 0$. Then, 
\begin{equation} \label{selrelation3.3}
    \dim_{\FF}\Sel(E^F/L) = \dim_{\FF}\Sel(E/L) - \dim_{\mathbb{F}_2}(V_T) + d
\end{equation} for some $d$ such that 
$$0 \leq d \leq \dimf \left(\bigoplus_{\mathfrak p \in T} H^1_f(L_\mathfrak{p}, E[2])/V_T \right)$$
and 
$$d \equiv \dimf \left(\bigoplus_{\mathfrak p \in T} H^1_f(L_\mathfrak{p}, E[2])/V_T \right) \pmod{2}.$$
\end{lemma}

An immediate consequence of the above is the following lemma.

\begin{lemma}[\cite{MR}*{Corollary 3.4}]\label{MR3.4}
For an elliptic curve $E/L$ and for $F/L$ and $T$ as defined in Lemma \ref{MR3.3}, we have:
\begin{enumerate}
    \item If $\dimf\left(\bigoplus_{\mathfrak p \in T} H^1_f(L_\mathfrak{p}, E[2])/V_T \right) \leq 1$ then 
    $$\dim_{\FF}\Sel(E^F/L)=\dim_{\FF}\Sel(E/L) - 2 \dimf V_T + \sum_{\mathfrak p \in T} \dimf H^1_f(L_\mathfrak{p}, E[2]).$$
    \item 
    If $T$ is empty, then $\dim_{\FF}\Sel(E^F/L) = \dim_{\FF}\Sel(E/L)$. 
\end{enumerate}
\end{lemma}

We will use Lemma \ref{MR3.4} setting $L$ to be some admissible cubic resolvent, $K_3$, to understand the $2$-Selmer rank of some square norm twists by controlling $\dim_{\FF}\mathrm{loc}_T (\Sel(E/K_3)$ and $\dim_{\FF} H^1_f((K_3)_\mathfrak{p}, E[2])$ for each $\mathfrak{p} \in T$.

\section{Twisting by Square Norm Extensions}

In this section we will consider elliptic curves $E/\mathbb{Q}$ together with some $K_3$ which will always be assumed to be an \emph{admissible cubic resolvent for $E$} as in Definition \ref{def:admissibleK3}.  Recall that among other conditions, we require $E(K_3)[2]=0$ and $K_3$ to $\mathbb{Q}(E[2])$ be linearly disjoint. 

We are concerned with quadratic twists $E^F$ over $K_3$ where we impose conditions on $F$. In Section 1 we introduced Definition \ref{def:sqtwist} defining \emph{square norm twists} to keep track of conditions on the twists. Recall that for an elliptic curve defined over a number field $L$, $E/L$, these are quadratic twists $E^F/L$ of $E/L$ where $F=L(\sqrt{\alpha})$, $\alpha \in L^\times/(L^\times)^2$, and $N_{L/\mathbb{Q}}(\alpha)$ is a square.

We will be interested in the application of the definition above where $L=K_3$, which, as above, will be the cubic resolvent for some quartic $S_4$ extensions of $\mathbb{Q}$. 

Further, define $N_r^\square(E, X)$ as follows to count quadratic extensions $F/K_3$  with bounded conductor, $\mathfrak{f}(F/K_3)$, that give square norm twists $E^F$ of $E$ with $2$-Selmer group of dimension $r$:
\begin{align*}
    N_r^\square(E,X) = \#\{ F=K_3(\sqrt{\alpha}) \mid&  \alpha \in K_3^\times/(K_3^\times)^2,  N_{K_3/\Q}(\alpha)~ \mathrm{a ~square},\\
    &\dim_{\FF}\Sel(E^F/K_3)=r, N_{K_3/\Q}\mathfrak{f}(F/K_3) < X\}.
\end{align*}
With that in mind, in this section we prove the following.

\begin{proposition}\label{prop:manytwists}
Fix an $S_3$
or $C_3$
cubic field $K_3/\mathbb{Q}$, an elliptic curve $E/\mathbb{Q}$, and a nonnegative even integer $r$. Suppose there exists a square norm twist, $E^L/K_3$, of $E/K_3$, with $\dim_{\FF}\mathrm{Sel}(E^L/K_3)=r$. Then we have:
\begin{itemize}
    \item If $\mathrm{Gal}(K_3/\mathbb{Q}) \simeq S_3$ and $\mathrm{Gal}(K_3(E[2])/K_3) \simeq S_3$, then
    $$ N_r^\square(E,X) \gg X^{1/2}/\log(X)^{5/6};$$
    \item if $\mathrm{Gal}(K_3/\mathbb{Q}) \simeq S_3$ and $\mathrm{Gal}(K_3(E[2])/K_3) \simeq C_3$, then 
    $$N_r^\square(E,X) \gg X^{1/2}/\log(X)^{2/3};$$
    \item if $\mathrm{Gal}(K_3/\mathbb{Q}) \simeq C_3$ and  $\mathrm{Gal}(K_3(E[2])/K_3) \simeq S_3$, then
    $$ N_r^\square(E,X) \gg X^{1/2}/\log(X)^{8/9};$$
    \item if $\mathrm{Gal}(K_3/\mathbb{Q}) \simeq C_3$ and $\mathrm{Gal}(K_3(E[2])/K_3) \simeq C_3$, then 
    $$N_r^\square(E,X) \gg X^{1/2}/\log(X)^{7/9}.$$
\end{itemize}
\end{proposition}

\begin{remark}
In the next section, we will prove the existence of the quadratic extension $L/K_3$ from the hypotheses of Proposition \ref{prop:manytwists}. With this, we will use the relationship between the rank growth from $\mathbb{Q}$ to $K$ and the rank growth $K_3$ to a quadratic extensions $F/K_3$ to prove Theorem \ref{thm:mainrank}. 
\end{remark}

The rest of this section will be devoted to proving Proposition \ref{prop:manytwists}. Before proceeding, we first need to enumerate ideals in $K_3$ that allow us to get quadratic extensions for square norm twists. We will then show for each such ideal, there is a square norm twist of $E$ corresponding to that ideal.

\begin{lemma}\label{lem:manyideals}
Suppose $E/\mathbb{Q} $ is an elliptic curve where $K_3$ is an admissible $S_3$-cubic resolvent for $E$ and  $\mathrm{Gal}(K_3(E[2]/K_3)$ is $C_3$ or $S_3$. Let $S$ be the set of the elements of order 3 in $\text{Gal}(K_3(E[2])/K_3)$, and  $N$ be a ray class field of $K_3$.
Then the number of ideals $\mathfrak b$ of $K_3$ such that 
\begin{itemize}
    \item $\nrm\mathfrak{b} < X$ and $[\mathfrak{b}, N/K_3]=1$,
    \item and for every prime ideal $\mathfrak{p}$ dividing $\mathfrak{b}$,  $\nrm\mathfrak{p}$ is a square and $\mathrm{Frob}_\mathfrak{p}(K_3(E[2])/K_3) \subset S$
\end{itemize}
is asymptotic to 
$$(C+o(1))\frac{X^{1/2}}{\log(X)^{1-\frac{1}{2}\frac{|S|}{[K_3(E[2]):K_3]}}}$$
as $X \to \infty$,
where $C$ is some positive constant,  $[-~, N/K_3]$ is the global Artin symbol and $\nrm$ is the ideal norm.
\end{lemma}

\begin{proof}
An unramified, non-inert rational prime $p$ can split as a product of primes in two ways in the ring of integers, $\mathcal{O}_{K_3}$, of  the cubic field $K_3$. Either $p\scr{O}_{K_3}=\mathfrak{p}_1\mathfrak{p}_2\mathfrak{p}_3$ where each factor has degree one, or $p\scr{O}_{K_3}=\mathfrak{p}_1\mathfrak{p}_2$ where one factor has degree one and one factor has degree two. Primes of degree two only appear as factors in the latter splitting type. 

First we will count rational primes $p$ such that $p\scr{O}_{K_3}=\mathfrak{p}\mathfrak{q}$ where the residue degrees of the prime factors are $f(\mathfrak{p}| p)=1$ and $f(\mathfrak{q}| p)=2$. For each such $p < X$ we get one prime $\mathfrak q$ of $K_3$ of square norm such that $\nrm\mathfrak{q} < X^2$. Let $\scr{S}_{(1,2)}$ be the set of such rational primes, i.e., 
\begin{equation}
\scr{S}_{(1,2)} \coloneqq \{ p \in \mathbb{N} ~ \text{prime} \mid p\scr{O}_{K_3}=\mathfrak{p}\mathfrak{q}, f(\mathfrak{p}| p)=1, f(\mathfrak{q}| p)=2 \}
\end{equation}
Also set
\begin{equation}
\scr{P}_{(1)} \coloneqq \{\mathfrak{p} \subset \mathcal{O}_{K_3} \text{ prime ideal}  \mid  (\mathfrak{p} \cap \mathbb{Z})\scr{O}_{K_3}=\mathfrak{p}\mathfrak{q}, f(\mathfrak{p}|\mathfrak{p} \cap \mathbb{Z})=1, f(\mathfrak{q} | \mathfrak{p}\cap \mathbb{Z})=2 \}
\end{equation}
and 
\begin{equation}
\scr{Q}_{(2)} \coloneqq \{\mathfrak{q} \subset \mathcal{O}_{K_3} \text{ prime ideal}  \mid f(\mathfrak{q}\mid \mathfrak{q}\cap \mathbb{Z}) =2 \}.
\end{equation}
Likewise, define
\begin{itemize}
 \item $\scr{S}_{(1,2)}(X) \coloneqq \{p \in \scr{S}_{(1,2)} \mid p <X \}$;
 \item $\scr{P}_{(1)}(X) \coloneqq \{\mathfrak p \in \scr{P}_{(1)} \mid \nrm \mathfrak p < X\}$;
 \item $\scr{Q}_{(2)}(X) \coloneqq \{\mathfrak q \in \scr{Q}_{(2)} \mid \nrm \mathfrak q < X\}$. 
 \end{itemize}
 With this notation, the discussion above amounts to 
\begin{equation}
\# \scr{S}_{(1,2)}(X) = \#\scr{P}_{(1)}(X) = \#\scr{Q}_{(2)}(X^2).
\end{equation}

A rational prime $p \in \scr{S}_{(1,2)}$ if and only if $\text{Frob}_p(\tilde K_3/\mathbb{Q})$ acts on the three cosets of $\text{Gal}(\tilde K_3/\mathbb{Q})/\text{Gal}(\tilde K_3/K_3)$ like a transposition. Via the Chebotarev density theorem, this happens with probability $\#\{\mathrm{transpositions ~in ~} S_3\}/\#S_3 =1/2$. That is, the density of $\scr{S}_{(1,2)}$ in the set of all rational primes is $1/2$.

We can conclude the Dirichlet density, $\delta_{dir}$, of the set of primes $\mathfrak{p}$ in $K_3$ corresponding to each $p \in \scr{S}_{(1,2)}$ is also 1/2. i.e.
\begin{equation}
\delta_{dir}\left(\{\mathfrak{p} \mid f(\mathfrak{p}| \mathfrak{p} \cap \mathbb{Z})=1, (\mathfrak{p} \cap \mathbb{Z})\scr{O}_{K_3}=\mathfrak{p}\mathfrak{q}, f(\mathfrak{q} | \mathfrak{p}\cap \mathbb{Z})=2 \} \right) = \delta_{dir}(\scr{P}_{(1)}) = 1/2.
\end{equation}
Now define some notation. Set $M = K_3(E[2])$ and recall that $S$ is the set of all elements of order 3 in $\text{Gal}(M/K_3)$ and note that $S$ is a union of conjugacy classes when $\text{Gal}(M/K_3) = C_3$ and is a conjugacy class when $\text{Gal}(M/K_3)=S_3$. Now, set

\begin{itemize}
\item $\scr{P} =  \{\mathfrak{p} \in \scr{P}_{(1)} \mid \mathfrak{p} \text{ unramified in } NM/K_3, \text{Frob}_{\mathfrak{p}}(M/K_3) \subset S\}$;

\item $\scr{Q} =  \{\mathfrak{q} \in \scr{Q}_{(2)} \mid \mathfrak{p} \text{ unramified in } NM/K_3, \text{Frob}_{\mathfrak{q}}(M/K_3)\subset S \}$; 

\item $\scr{N} = \{\mathfrak{a} \mid \text{square-free product of ideals from } \scr{P}\}$;

\item $\scr{N}_1 = \{\mathfrak{a} \mid \text{square-free product of ideals from } \scr{P}, [\mathfrak{a}, N/K_3]=1\}$;

\item $\scr{R}_1 = \{\mathfrak{b} \mid \text{square-free product of ideals from } \scr{Q}, [\mathfrak{b}, N/K_3]=1\}$.
\end{itemize}
Our goal is now to access the number of ideals in $\scr{N}_1(X)$ via the Dirichlet series $\sum_{\mathfrak{a} \in \scr{N}_1}N\mathfrak{a}^{-1}$ and a Tauberian theorem of Wintner. Indeed, we'll see knowing $\scr{N}_1(X)$ suffices to understand $\scr{R}_1(X^2)$.

To that end, for an irreducible character $\chi: \text{Gal}(N/K_3) \rightarrow \mathbb{C}^\times$ where we will write $\chi(\mathfrak{a})$ for $\chi([\mathfrak{a}, N/K_3])$, set 
\begin{equation}\label{def:fchi}
f_\chi(s) := \sum_{\mathfrak{a} \in \scr{N}} \chi(\mathfrak{a})N\mathfrak{a}^{-s} = \prod_{\mathfrak{p} \in \scr{P}} (1 + \chi(\mathfrak{p})N\mathfrak{p}^{-s}).
\end{equation}

Note that $\mathfrak{p} \in \scr{P}$ can't be above a rational prime $p$ which splits completely in $\tilde K_3$; if it split completely in $\tilde K_3$, then it splits completely in $K_3$, too. Thus $\mathrm{Frob}_{\mathfrak p}(\tilde K_3/K_3)$ isn't trivial.

Let $\tau$ be the non-trivial element of $\mathrm{Gal}(\tilde K_3/K_3)$ and set  
$$S' = \{\tau\}\times S \subset \mathrm{Gal}(\tilde K_3 M /K_3) = \mathrm{Gal}(\tilde K_3/K_3) \times \mathrm{Gal}(M/K_3)$$
and 
$$\delta(S,\chi) = \begin{cases}
0 \text{ if } \chi \text{ non-trivial} \\
\frac{1}{2}\frac{|S|}{[M:K_3]} \text{ if } \chi \text{ trivial}
\end{cases}$$
noting that, in the $\chi$ trivial case, 
$$
\frac{1}{2}\frac{|S|}{[M:K_3]} = \#S' / \#\mathrm{Gal}(\tilde K_3 M/K_3).
$$
We write $g_1(s) \sim g_2(s)$ for two complex functions $g_1,g_2$ on the half plane $\Re s > 1$ if $g_1(s) - g_2(s)$ extends to a holomorphic function on the half plane $\Re s \geq 1$. Now, starting from the logarithm of \eqref{def:fchi} and using the Chebotarev density theorem, we have
\begin{align*}
\log f_\chi(s) &\sim \sum_{\mathfrak{p} \in \scr{P}}\chi(\mathfrak{p})N\mathfrak{p}^{-s} \sim \delta(S,\chi) \sum_{\substack{\mathfrak{p} \text{ prime}}}\chi(\mathfrak{p})N\mathfrak{p}^{-s} \sim \delta(S, \chi) \log\left(\frac{1}{s-1}\right).
\end{align*}

Using character orthogonality, observe
\begin{equation}\label{eq:poles}
\frac{1}{[N:K_3]}\sum_{\chi}f_\chi(s) = \frac{1}{[N:K_3]}\sum_{\mathfrak{a} \in \scr{N}} N\mathfrak{a}^{-s}\sum_{\chi}\chi(\mathfrak{a}) = \sum_{\mathfrak{a} \in \scr{N}_1} N\mathfrak{a}^{-s} = 
(s-1)^{-\frac{1}{2}\frac{|S|}{[M:K_3]}}h(s)
\end{equation}
where the first two sums range over irreducible characters $\chi$ of $\text{Gal}(N/K_3)$, and where $h(s)$ is a non-zero, holomorphic function for $\Re s \geq 1$.

Applying a Tauberian theorem of Wintner \cite{Wintner} to \eqref{eq:poles}, we obtain 

\begin{equation}\label{eq:sizeofN1}
\#\scr{N}_1(X) = (C+o(1))\frac{X}{\log(X)^{1-\frac{1}{2}\frac{|S|}{[M:K_3]}}}.
\end{equation}

Now, if $\mathfrak{a} \in \scr{N}_1$ we have, for some positive integer $m$, $\mathfrak{a} = \prod_{i=1}^m \mathfrak{p}_i$,  where $\mathfrak{p}_i \in \scr{P}$. For each rational prime $p_i$ below $\mathfrak{p}_i$, we have $p_i\scr{O}_{K_3} = \mathfrak{p}_i\mathfrak{q}_i$ where $\mathfrak{q}_i \in \scr{Q}$. Set $\mathfrak{b}=\prod_{i=1}^m \mathfrak{q}_i$.

First, we'll show $\text{Frob}_{\mathfrak{q_i}}(M/K_3) \subset S$. If $E: y^2=f(T)$, consider the 
cubic extension $L=K_3(T)/(f(T))$, where $f(T) \in \mathbb{Q}[x]$ is some cubic polynomial, between $M$ and $K_3$. If $\text{Gal}(M/K_3) = C_3$, then $M=L$. We can look at how $f(T)$ factors modulo $\mathfrak{p}_i$ and $\mathfrak{q}_i$. The only way for $f(T)$ to be irreducible modulo $\mathfrak{q}_i$ (i.e. over $\mathbb{F}_{p_i^2})$ is for $f(T)$ to be irreducible modulo $\mathfrak{p}_i$; this  happens precisely when $\text{Frob}_{\mathfrak{p}_i}(M/K_3) \subset S$. If $f(T) \pmod{\mathfrak{q}_i}$ is irreducible, $\text{Frob}_{\mathfrak{q}_i}(M/K_3)$ has order 3.  That is, demanding $\text{Frob}_{\mathfrak{p}_i}(M/K_3) \subset S$ forces $\text{Frob}_{\mathfrak{q}_i}(M/K_3) \subset S$.  

Second, since each $\mathfrak{p}_i\mathfrak{q}_i$ is principle, knowing $[\mathfrak{a}, N/K_3]=1$ suffices to show $[\mathfrak{b}, N/K_3]=1$, too. Thus, $\mathfrak{b} \in \scr{R}_1$ 

Finally, since $N\mathfrak{b} = (N\mathfrak{a})^2$, we have established a bijection between $\scr{N}_1(X)$ and $\scr{R}_1(X^{1/2})$ by mapping $\mathfrak{a} \mapsto \mathfrak{b}$.

This and \eqref{eq:sizeofN1} give us 
$$\#\scr{R}_1(X) \sim (C+o(1))\frac{X^{1/2}}{\log(X)^{1-\frac{1}{2}\frac{|S|}{[M:K_3]}}}$$
for some positive constant $C$, as needed.
\end{proof}

We now state and prove the analogue of Lemma \ref{lem:manyideals}
in the case that $K_3$ is an admissible $C_3$-cubic resolvent. 

\begin{lemma}\label{lem:manyidealsA4}
Suppose $E/\mathbb{Q} $ is an elliptic curve where $K_3$ is an admissible $C_3$-cubic resolvent for $E$ and  $\mathrm{Gal}(K_3(E[2]/K_3)$ is $C_3$ or $S_3$.  Let $S$ be the set of the elements of order 3 in $\text{Gal}(\mathbb{Q}(E[2])/\mathbb{Q})$, and $N$ be an abelian extension of $K_3$. Then, the number of ideals $\mathfrak{b}$ of $K_3$ such that 
\begin{itemize}
	\item $\nrm\mathfrak{b} < X$, $\nrm\mathfrak{b}$  is a square, and $[\mathfrak{b}, N/K_3]=1$; 
    \item and for every prime ideal $\mathfrak{p}$ dividing $ \mathfrak{b}$,   $\mathrm{Frob}_\mathfrak{p}(K_3(E[2])/K_3) \subset S$ 
\end{itemize}
is asymptotic to 
$$(D+o(1))\frac{X^{1/2}}{\log(X)^{1-\frac{1}{2}\frac{|S|}{[K_3(E[2]):\mathbb{Q}]}}}$$
for some real, positive constant $D$, and where 
$[-~, N / K_3]$ is the global Artin symbol and $\nrm$ is the ideal norm.
\end{lemma}

\begin{proof}
Since $K_3$ is an admissible $C_3$-cubic resolvent, we have that $K_3$ and $\mathbb{Q}(E[2])$ are linearly disjoint. 
Setting $M = K_3(E[2])$, we have $\mathrm{Gal}(M/\mathbb{Q}) = C_3 \times S_3$. 
First, some notation. Define the following sets, 
\begin{itemize}
    \item $\mathcal{P_\mathbb{Q}} = \{p \in \mathbb{N} ~\text{prime} \mid \mathrm{Frob}_p(M/\mathbb{Q})\subset \{1\} \times S, p ~\text{unramified in}~ \tilde N K_3 \mathbb{Q}(E[2])\}$;
    \item $\mathcal{A} = \{a \mid a ~\text{a squarefree product of }~ p \in \mathcal{P}_\mathbb{Q}\}$;
    \item $\mathcal{A}_1 = \{a \in \mathcal{A} \mid [(a), \tilde N/\mathbb{Q}]=1 \}$ 
    where $[-~, \tilde N / \mathbb{Q}]$ is the global Artin symbol.
\end{itemize}

Let $\tilde N$ be the normal closure of $N$ over $\mathbb{Q}$. We will use the triviality of the Artin symbol $[-~, \tilde N/\mathbb{Q}]$ to obtain the triviality of the Artin symbol $[-~, N/\mathbb{Q}]$ as in the statement of the lemma. Now, for a character $\psi : \mathrm{Gal}(\tilde N/\mathbb{Q}) \rightarrow \mathbb{C}^\times$, and writing $\psi(a)$ for $\psi([(a), \tilde N/\mathbb{Q}])$, let 
\begin{equation*}
    f_\psi(s) \coloneqq \sum_{a \in \mathcal{A}}\psi(a)a^{-s}
\end{equation*}
For two functions $g_1, g_2$ defined on the complex half plane $\Re s >1$, write $g_1(s) \sim g_2(s)$ to mean $g_1(s)$ and $g_2(s)$ differ by a function which is holomorphic on $\Re s \geq 1$.  Taking log of the $f_\psi(s)$, substituting the Taylor series for $\log(1-x)$ and truncating the Taylor series after one term, one arrives at 
\begin{equation}\label{eq:logfpsi}
    \log f_\psi(s) = \sum_{p \in \mathcal{P}_\mathbb{Q}} \log\left(1 + \psi(p)p^{-s}\right) \sim \sum_{p \in \mathcal{P}_{\mathbb{Q}}} \psi(p)p^{-s}\sim \delta_\psi \log\left(\frac{1}{s-1}\right)
    \end{equation}
where, using the Chebotarev density theorem, 
$$
\delta_\psi = 
\begin{cases} 0 \hspace{1.6cm}\text{if $\psi$ is non-trivial}, \\ 
\frac{|S|}{[M:\mathbb{Q}]} ~\text{if $\psi$ is trivial}.
\end{cases}
$$

Now, using character orthogonality and summing over irreducible characters $\psi$ of $\mathrm{Gal}(\tilde N/\mathbb{Q})$, we have
\begin{equation}\label{eq:A1ortho}
\frac{1}{[\tilde N : \mathbb{Q}]} \sum_{\psi} f_\psi(s) =  \frac{1}{[\tilde N : \mathbb{Q}]} \sum_{a \in \mathcal{A}} a^{-s} \sum_{\psi} \psi(a) = \sum_{a \in \mathcal{A}_1}a^{-s}.
\end{equation}
But also, using \eqref{eq:logfpsi}, we have 
\begin{equation}\label{eq:expCheby}
    \frac{1}{[\tilde N : \mathbb{Q}]} \sum_{\psi} f_\psi(s) = g(s)(s-1)^{-|S|/[M:\mathbb{Q}]}
\end{equation}
where $g(s)$ is holomorphic and non-zero on $\Re s \geq 1$. We have then, via \eqref{eq:A1ortho} and \eqref{eq:expCheby}, we have 
$$ \sum_{a \in \mathcal{A}_1}a^{-s} =  g(s)(s-1)^{-|S|/[M:\mathbb{Q}]}.$$
Applying a Tauberian theorem of Wintner \cite{Wintner} yields 
$$\#\{a \in \mathcal{A}_1 \mid a < X \} = (C + o(1))\frac{X}{(\log X)^{1-|S|/[M:\mathbb{Q}]}}$$
for some positive, real $C$.

Now, suppose $a \in \mathcal{A}_1$ and $a = \prod_{i=1}^r p_i$, where the $p_i$ are distinct primes and $p_i \in \mathcal{P}_{\mathbb{Q}}$ and each $p_i$ splits completely in $\mathcal{O}_{K_3}$. Set $\mathfrak{a} = a\mathcal{O}_{K_3}$. 
Then $\mathfrak a$ decomposes into prime ideals as $\mathfrak{a} = \prod_{i=1}^r \mathfrak{p}_i\mathfrak{p}_i' \mathfrak{p}_i''$ where $\mathfrak{p}_i, \mathfrak{p}_i'$, and $\mathfrak{p}_i''$ are the three primes above $p_i$. For each each $\mathfrak{p}_i$ above a prime $p_i$, pick another prime $\mathfrak{p}'_i$ of $K_3$ above $p_i$ (there are two choices), and set $\mathfrak{b} = \prod_{i=1}^r \mathfrak{p}_i\mathfrak{p}'_i$. Note $\nrm_{K_3/\mathbb{Q}}\mathfrak{b} = a^2$. 

In this way, counting $a \in \mathcal{A}_1$ with $a < X$ gives a way of counting ideals $\mathfrak b$ in $K_3$ such that $\nrm_{K_3/\mathbb{Q}}\mathfrak b < X^{1/2}$ such that $\nrm_{K_3/\mathbb{Q}}\mathfrak b$ is a square and, for each prime $\mathfrak p$ dividing $\mathfrak{b}$, $\mathrm{Frob}_{\mathfrak p}(M/K_3) \subset S$ and $[\mathfrak{b}, N/K_3]=1$. The lemma follows.
\end{proof}

For an elliptic curve $E$ and each of the ideals enumerated in Lemma \ref{lem:manyideals}, there is a twist of $E$ in which the $2$-Selmer rank remains the same. 

\begin{lemma}\label{lem:sameranktwist}
Keeping the notation of Lemma \ref{lem:manyideals}, if $\mathfrak{b}$ is an ideal of $K_3$ such that 
\begin{itemize}
    \item $\nrm\mathfrak{b} < X$,
    \item If a prime ideal $\mathfrak{p}$ divides $\mathfrak{b}$, then $\nrm\mathfrak{p}$ is a square, 
    \item $\text{Frob}_\mathfrak{p}(K_3(E[2])/K_3) \subset S$, 
    \item and $[\mathfrak{b}, N/K_3]=1$,
\end{itemize}
then there is a quadratic extension $F/K_3$ of conductor $\mathfrak{b}$ such that $\dim_{\FF}\Sel(E^F/K_3) = \dim_{\FF}\Sel(E/K_3)$.
\end{lemma}
\begin{proof}
This is Proposition 4.2 of \cite{MR}, with $N=K_3(8\Delta_E\infty)$, the ray class field of $K_3$ modulo $8\Delta_E$ and all archimedean places of $K_3$, applied to the relevant ideals, which are a subset of the ideals discussed covered in that result.
\end{proof}

For an elliptic curve $E$ and each of the ideals enumerated in Lemma \ref{lem:manyidealsA4}, there is a twist of $E$ in which the $2$-Selmer rank remains the same. 

\begin{lemma}\label{lem:sameranktwistA4}
Keeping the notation of Lemma \ref{lem:manyidealsA4}, if $\mathfrak{b}$ is an ideal of $K_3$ such that 
\begin{itemize}
    \item $\nrm\mathfrak{b} < X$,
    \item If a prime ideal $\mathfrak{p}$ divides $\mathfrak{b}$, then $\nrm\mathfrak{p}$ is a square, 
    \item $\text{Frob}_\mathfrak{p}(K_3(E[2])/K_3) \subset S$, 
    \item and $[\mathfrak{b}, N/K_3]=1$,
\end{itemize}
then there is a quadratic extension $F/K_3$ of conductor $\mathfrak{b}$ such that $\dim_{\FF}\mathrm{Sel}_2(E^F/K_3) = \dim_{\FF}\mathrm{Sel}_2(E/K_3)$.
\end{lemma}
\begin{proof}
The proof is the same as that of Lemma \ref{lem:sameranktwist}.
\end{proof}

We are now ready to prove the main result of the section. We follow exactly the strategy of the proof of \cite{MR}*{Theorem 1.4} with the additional step of keeping track of the square norm condition of the involved quadratic twists. 

\begin{proof}[Proof of Proposition \ref{prop:manytwists}]
As in  Lemmas \ref{lem:manyideals} and \ref{lem:manyidealsA4}, let $S$ be the set of order 3 elements in $\Gal(K_3(E[2])/K_3)$. Then if $\Gal(K_3(E[2])/K_3) \simeq S_3$, 
$$\frac{|S|}{[K_3(E[2]) : K_3]} = 1/3,$$
and if  $\Gal(K_3(E[2])/K_3) \simeq \Z/3\Z$,
$$\frac{|S|}{[K_3(E[2]) : K_3]} = 2/3.$$
We'll consider the case that $K_3$ is an admissible $S_3$-cubic resolvent; there are two sub-cases to consider:
\begin{enumerate}
    \item Suppose $\dim_{\FF}\Sel(E/K_3)=r$. By Lemma \ref{lem:sameranktwist} and Lemma \ref{lem:manyideals}, the number of square norm twists $E^F/K_3$ such that $\dim_{\FF}\Sel(E^F/K_3)=r$ is 
    $$\gg X^{1/2}/\log{X}^{1-\frac{1}{2}\frac{|S|}{[K_3(E[2]):K_3]}}.$$
    \item Suppose $\dim_{\FF}\Sel(E/K_3) \neq r$. We have assumed there is a square norm twist $E^L/K_3$ such that $\dim_{\FF}\Sel(E^L/K_3) = r$. Note that a square norm twist of a square norm twist results in the square norm twist. That is, a is a square norm twist  $(E^L)^{F'}$ of $E^L$ is itself a square norm twist $E^F$ of $E$. Now the result follows from Case (1) applied to $E^L$. 
\end{enumerate}
If instead $K_3$ is an admissible $C_3$-cubic resolvent for $E$, the proof is the same as above, but with Lemmas \ref{lem:manyidealsA4} and \ref{lem:sameranktwistA4} in place of Lemmas \ref{lem:manyideals} and \ref{lem:sameranktwist}, respectively. 
\end{proof}

\section{Decreasing the 2-Selmer Rank}\label{section:rankdecrease}
Our strategy will be to use Lemma \ref{MR3.4}(2) to understand the $2$-Selmer rank of square norm twists. We'll then use Proposition \ref{prop:manytwists} to show there are many square norm twist with prescribed 2-Selmer rank assuming we \textit{have already} a square norm twist of that prescribed 2-Selmer rank. In this section, we'll show a square norm twist with that prescribed 2-Selmer rank \textit{must exist} by showing we can take square norm twists that reduce the 2-Selmer rank by two; this is the content of Proposition \ref{lem:selrank-2}.  

Indeed, Proposition \ref{lem:selrank-2} below can be viewed as analogous to Proposition 5.1(iii) of \cite{MR}. Except, instead of decreasing the 2-Selmer rank by 1 via a quadratic twist obtained by controlling one local condition, we decrease the 2-Selmer rank by 2 via a square norm twist obtained by controlling two local conditions. Those two local conditions are obtained from two primes in the cubic resolvent above the same rational prime. 

\begin{proposition}\label{lem:selrank-2}
Let $E$ be an elliptic curve defined over $\mathbb{Q}$ and let $K_3$ be an admissible cubic resolvent for $E$. Suppose further that $\dim_{\FF}\Sel(E/\Q)=0$ and 
$\dim_{\mathbb{F}_2} \Sel (E/K_3) \geq 2$. 

Then there exists a square norm twist $E^F/K_3$ such that 
\[\dim_{\mathbb{F}_2} \mathrm{Sel}_2(E^F/K_3) = \dim_{\mathbb{F}_2} \mathrm{Sel}_2(E/K_3) - 2.\]
\end{proposition}

\begin{proof}
Let $\Delta_E$ be the discriminant of (a minimal model of) $E$. For the admissible cubic resolvent, $K_3$, of $E$, set $M = K_3(E[2])$. Note that $M/K_3$ is a Galois $S_3$-extension, since $K_3$ is as in Definition \ref{def:admissibleK3}, and $K_3(\sqrt{\Delta_E})/K_3$ is an intermediate quadratic extension. Let $v_0$ be the distinguished place of $K_3$ guaranteed by Definition \ref{def:admissibleK3}.

Let $\Sigma$ be a finite set that contains the following places of $K_3$: all infinite places of $K_3$, places of bad reduction for $E$,  and all primes above 2. Now set $\mathfrak{d} = \prod_{v \in \Sigma \setminus \{v_0\}} v$.  Let $K_3(8\mathfrak{d})$ be the ray class field of $K_3$ with modulus $8\mathfrak{d}$ and let $K_3[8\mathfrak{d}]$ be the maximal extension between $K_3(8\mathfrak{d} )$ and $K_3$ whose degree is a power of 2.

Let $\tilde K_3$ be the Galois closure of $K_3$ over $\mathbb{Q}$ and $\mathfrak{D} = \prod_{\substack{V \mid v \\ v \mid \mathfrak{d}}} V$ be the product of places of $\tilde K_3$ above the places of $K_3$ dividing $\mathfrak{d}$. Let $L$ be the maximal extension\footnote{Note that if $K_3$ is $C_3$-cubic, then $L = K_3[8\mathfrak{d}]$ since $\tilde K_3 = K_3$.} of $\tilde K_3$ between $\tilde K_3$ and the ray class field of  $\tilde K_3$ with modulus $8\mathfrak{D}$. Note $L \supseteq K_3[8 \mathfrak d]$, $[L:K_3]$ is a power of 2, and  $L$ is Galois over $\mathbb{Q}$. 

By assumption, $\tilde K_3$ and $\mathbb{Q}(E[2])$ are linearly disjoint over $\mathbb{Q}$.  $K_3[8\mathfrak{d}]$ and $M$ are linearly disjoint as extensions of $K_3$ since $v_0$ is ramified in $K_3(\sqrt{\Delta_E})$ but not in $K_3[8\mathfrak{d}]$. By the conditions on $v_0$ of Definition \ref{def:admissibleK3}, $v_0$ does not ramify from $K_3$ to $\tilde K_3$, and hence is unramified in $L$. So likewise, the same consideration of $K_3(\sqrt{\Delta_E})$ shows $M$ and $L$ are linearly disjoint over $K_3$.

Let $\sigma$ be an element of the absolute Galois group of $K_3$ such that $\sigma|_M$ is a transposition in $\mathrm{Gal}(M/K_3) \simeq \mathrm{Aut}(E[2])$ and $\sigma|_L=1$. The former condition implies $E[2]/(\sigma -1)E[2] \simeq \mathbb{F}_2$. The latter condition implies $\sigma|_{K_3[8\mathfrak{d}]}=1$.
 
For the remainder of the proof, fix a nonzero map $\phi: \mathrm{Sel}_2(E/K_3) \to E[2]/(\sigma -1)E[2]$. By \cite{MR}*{Lemma 3.5}, there is an element $\gamma \in G_{K_3}$ for which $\gamma = \sigma$ when restricted to $MK_3[8\mathfrak{d}]$ and $c(\gamma) = \phi(c)$ for all $c \in \mathrm{Sel}_2(E/K_3)$. 

Let $N$ be a Galois extension of $\mathbb{Q}$ containing $M$ and $L$ for which the restriction of  $\mathrm{Sel}_2(E/K_3)$ to $N$ is zero. For instance, take $N$ to be the Galois closure (over $\mathbb{Q}$) of the compositum of $M$, $L$, and the fixed field of the kernel of the restriction to $\mathrm{Hom}(G_M, E[2])$ of every $c \in \mathrm{Sel}_2(E/K_3)$. 

Let $\mathcal{P}_\gamma$ be the set of primes $\mathfrak{p}$ of $K_3$ for which $\mathfrak{p} \notin \Sigma$ and $\mathrm{Frob}_{\mathfrak{p}}( N/K_3) = \gamma|_N$. By the Cebotarev Density Theorem, the natural density of $\mathcal{P}_\gamma$ among the primes of $K_3$ is positive. i.e.,
$$\delta_\gamma := \lim_{X \to \infty}\frac{\#\{\mathrm{primes}~\mathfrak{p}~\mathrm{of}~K_3 \mid \mathrm{Frob}_{\mathfrak{p}}( N/K_3)=\gamma|_N, \mathfrak{p} \notin \Sigma, \mathbf{N}\mathfrak p < X\} }{\#\{\mathrm{primes}~\mathfrak{p}~\mathrm{of}~K_3 \mid \mathbf{N}\mathfrak p < X\}} > 0.$$

Now, let $\mathcal{P}_{sp}$ be the set of primes of $K_3$ with inertia degree one over $\mathbb{Q}$, i.e. the primes $\mathfrak{p}$ for which the rational prime ideal $\mathfrak{p} \cap \mathbb{Z}$ splits completely in $K_3$. Recall that $\mathcal{P}_{sp}$ has natural density one among the primes of $K_3$. In particular, this means we can pick a prime $\mathfrak{p}_1 \in \mathcal{P}_\gamma \cap \mathcal{P}_{sp}$. If we could not, then $\mathcal{P}_\gamma$ (which has positive density) would be contained in the complement of $\mathcal{P}_{sp}$ (which has density zero).    Let $p$ be the rational prime below $\mathfrak{p}_1$, and let $\mathfrak{p}_2$ and $\mathfrak{p}_3$ be the other two primes of $K_3$ above $p$, i.e. $p \mathcal{O}_{K_3}= \mathfrak{p}_1\mathfrak{p}_2\mathfrak{p}_3$. 

Our goal is now to construct a suitable square-norm twist from two of $\mathfrak{p}_1, \mathfrak{p}_2, \mathfrak{p}_3$. We can  understand the 2-Selmer group of this twist using Lemma \ref{MR3.4}(1), which requires us to compute both $H_f^1((K_3)_{\mathfrak{p}_i}, E[2])$ and $\mathrm{loc}_{\mathfrak{p}_i} \mathrm{Sel}_2(E/K_3)$. 

 First, consider the localization at $\mathfrak{p}_1$.  Since $\mathrm{Frob}_{\mathfrak{p}_1} = \gamma$ when restricted to $N$ (and $\sigma = \gamma$ when restricted to $MK_3[8\mathfrak{d}]$) we have both that 
\begin{equation}\label{eq:dim1restricted}
H^1_f((K_3)_{\mathfrak{p}_1}, E[2]) \simeq E[2]/(\sigma -1)E[2] \simeq \mathbb{F}_2
\end{equation} 
and $\phi(c) = c(\gamma)$ for all $c \in \mathrm{Sel}_2(E/K_3)$. The localization map 
$$\mathrm{loc}_{\mathfrak{p}_1}:\mathrm{Sel}_2(E/K_3) \to H^1_f((K_3)_{\mathfrak{p}_1}, E[2]) \simeq E[2]/(\mathrm{Frob}_{\mathfrak{p}_1} - 1)E[2] \simeq E[2]/(\sigma - 1)E[2] \simeq \mathbb{F}_2$$ 
is given by evaluation of cocycles at $\mathrm{Frob}_{\mathfrak{p}_1}$, so we can identify $\mathrm{loc}_{\mathfrak{p}_1} (\mathrm{Sel}_2(E/K_3)) = \phi(\mathrm{Sel}_2(E/K_3)) $ as subspaces of $\mathbb{F}_2$. Since $\phi$ is nonzero, 
\begin{equation}\label{eq:dim1prime}
\dim_{\mathbb{F}_2} \mathrm{loc}_{\mathfrak{p}_1}(\mathrm{Sel}_2(E/K_3)) = 1.
\end{equation}

It remains to understand the localizations at $\mathfrak{p}_2$ and $\mathfrak{p}_3$. Since $p$ splits completely in $K_3$, there is an equality of local fields $(K_3)_{\mathfrak{p}_1}  = (K_3)_{\mathfrak{p}_2}  =(K_3)_{\mathfrak{p}_3}   = \mathbb{Q}_p$ and so, together with \eqref{eq:dim1restricted} we have 
\begin{equation}\label{eq:equalH1f} 
H^1_f((K_3)_{\mathfrak{p}_1}, E[2]) = H^1_f((K_3)_{\mathfrak{p}_2}, E[2]) = H^1_f((K_3)_{\mathfrak{p}_3}, E[2])  = H_f^1(\mathbb{Q}_p, E[2]) \simeq \mathbb{F}_2.
\end{equation}

Beginning from the following commutative diagram, we will consider the localization of $\mathrm{Sel}_2(E/K_3)$ at primes of $K_3$ above $p$ and localization at $p$, 
\begin{center}
\begin{tikzcd}
{H^1(K_3, E[2])} \arrow[d, "\mathrm{cores}_{K_3/\mathbb{Q}}"'] \arrow[rr, "\oplus_{i=1}^3 \mathrm{loc}_{\mathfrak{p}_i}"] &  & {\oplus_{i=1}^3 H^1\left((K_3)_{\mathfrak{p}_i}, E[2]\right)} \arrow[d, "\oplus_{i=1}^3 \mathrm{cores}_{(K_3)_{\mathfrak{p}_i}/\mathbb{Q}_p}"'] \\
{H^1(\mathbb{Q}, E[2])} \arrow[rr, "\mathrm{loc}_{\{p\}}"]                                         &  & {H^1(\mathbb{Q}_p, E[2])}                                                                        
\end{tikzcd}
\end{center}
where the vertical map on the left side, 
$$\mathrm{cores}_{K_3/\mathbb{Q}}: H^1(K_3, E[2]) \to H^1(\mathbb{Q}, E[2]),$$
 is determined by corestriction  on Galois cohomology induced by the norm map \break  $\mathbf{N}_{K_3/\mathbb{Q}}:K_3 \to \mathbb{Q}$ (see \cite{MilneCFT}*{Example 1.29} or \cite{SerreGalCo}). The vertical map on the right side is the sum of corestriction maps: 
 $$\oplus_{i=1}^3 \mathrm{cores}_{(K_3)_{\mathfrak{p}_i}/\mathbb{Q}_p}: \bigoplus_{i=1}^3 H^1\left((K_3)_{\mathfrak{p}_i}, E[2]\right) \to H^1(\mathbb{Q}_p, E[2]) \quad \text{by} \quad (c_1, c_2, c_3) \mapsto \sum_{i=1}^3  \mathrm{cores}_{(K_3)_{\mathfrak{p}_i}/\mathbb{Q}_p}(c_i) .$$
 
 Restricting the left column to the 2-Selmer groups of $E$ over $K_3$ and $\mathbb{Q}$ and to restricted cohomology on the right column, together with \eqref{eq:equalH1f}, we have
\begin{center}
\begin{tikzcd}
\mathrm{Sel}_2(E/K_3) \arrow[d, "\mathrm{cores}_{K_3/\mathbb{Q}}"'] \arrow[rr, "\oplus_{i=1}^3 \mathrm{loc}_{\mathfrak{p}_i}"] &  & {\bigoplus_{i=1}^3 H_f^1((K_3)_{\mathfrak{p}_i}, E[2])} \arrow[d, "\oplus_{i=1}^3 \mathrm{cores}_{(K_3)_{\mathfrak{p}_i}/\mathbb{Q}_p}"'] \arrow[r, "\simeq"] & {H_f^1(\mathbb{Q}_p, E[2])^{\oplus 3} \simeq \mathbb{F}_2^3} \arrow[ld, "v_1 + v_2 + v_3"] \\
\mathrm{Sel}_2(E/\mathbb{Q})=0 \arrow[rr, "\mathrm{loc}_p"]                                                   &  & {H^1_f(\mathbb{Q}_p, E[2]) \simeq \mathbb{F}_2}                                                           &                                                                                           
\end{tikzcd}
\end{center}

where the diagonal map $\mathbb{F}_2^3 \to \mathbb{F}_2$ on the right is coordinate-wise addition of vectors in $\mathbb{F}_2^3$ modulo 2. For $c \in \mathrm{Sel}_2(E/K_3)$ we have $\mathrm{loc}_p\mathrm{cores}(c) = 0$  since $\mathrm{Sel}_2(E/\mathbb{Q})=0$. Hence, 
\begin{equation}\label{eq:locrelation}
 \mathrm{loc}_{\mathfrak{p}_1}(c) + \mathrm{loc}_{\mathfrak{p}_2}(c)  + \mathrm{loc}_{\mathfrak{p}_3}(c)  = 0 \quad \text{in}~ \mathbb{F}_2.
 \end{equation}
 
By \eqref{eq:dim1prime} there is an element $c \in \mathrm{Sel}_2(E/K_3)$ for which $\mathrm{loc}_{\mathfrak{p}_1}(c)=1$ viewed in $\mathbb{F}_2$. 
 Combining this with  \eqref{eq:locrelation}, there is exactly one prime  $\mathfrak{p}_i \in \{\mathfrak{p}_2, \mathfrak{p}_3\}$ for which $\mathrm{loc}_{\mathfrak{p}_i}(c) = 1$; suppose, without loss of generality,  it is $\mathfrak{p}_2$. Whence, 
\begin{equation}\label{eq:locp2dim}
\dim_{\mathbb{F}_2}\mathrm{loc}_{\mathfrak{p}_2} \mathrm{Sel}_2(E/K_3) = 1.
\end{equation}

Finally, we will twist $E/K_3$ by a quadratic extension $F/K_3$ ramified only at $\mathfrak{p}_1$ and $\mathfrak{p}_2$ to get our desired result. 

Let $\mathfrak{P}$ be a prime of $L$ above $\mathfrak{p}_1$. Since $L/\mathbb{Q}$ is Galois, we have 
$$  \mathrm{Frob}_\mathfrak{P}(L/\mathbb{Q})^{f(\mathfrak{p}_1/p)} =  \mathrm{Frob}_\mathfrak{P}(L/\mathbb{Q}) = \mathrm{Frob}_\mathfrak{P}(L/K_3)  = 1 $$
i.e., $p$ splits completely in $L$, and so $p$ splits completely in $K_3[8\mathfrak{d}]$. 

Since our choice of Frobenius class for $p$ is trivial when restricted just to $K_3[8\mathfrak d]$, \break $\mathrm{Frob}_{\mathfrak{p}_1}(K_3[8\mathfrak{d}]/K_3) =  \mathrm{Frob}_{\mathfrak{p}_2}(K_3[8\mathfrak{d}]/K_3) =  1$, and since $[K_3(8\mathfrak d) : K_3[8\mathfrak d]]$ is odd, there will be an odd integer, say $h$, such that $(\mathfrak{p}_1 \mathfrak{p}_2)^h$ is principal with generator $\alpha$ such that $\alpha \equiv 1 \pmod{8 \mathfrak d}$ and $\alpha$ positive at all real embeddings except possibly $v_0$.  

We are now in a position to construct the quadratic extension $F/K_3$ by which we will twist $E$: define $F=K_3(\sqrt{\alpha})$.  Only the primes $\mathfrak{p}_1$ and $\mathfrak{p}_2$ of $K_3$ ramify in $F$.  Set $p(x) = x^2 - \alpha$ and note  $p'(1)^2=4$. For any $\mathfrak{q} \in \Sigma \setminus \{v_0\}$, since $\alpha \equiv 1 \pmod{8 \mathfrak d}$, we also have $p(1) = 1 - \alpha \equiv 0 \pmod{4\mathfrak{q}}$. From Hensel's lemma (see e.g. \cite{Eisenbud}*{Theorem 7.3} for an applicable statement), it follows $p(x)$ has a root in $(K_3)_{\mathfrak{q}}$. So $(K_3)_{\mathfrak{q}} \otimes F = (K_3)_{\mathfrak{q}}^2$, i.e. $\mathfrak{q}$ splits in $F$. Thus all primes in $\Sigma \setminus \{v_0\}$ split in $F$.

Note also that $\textbf{N}(\alpha) = \textbf{N}(\mathfrak{p}_1)\textbf{N}(\mathfrak{p}_2)=p^{2h}$, so the quadratic twist $E^F/K_3$ of $E/K_3$ is a square norm twist.

Finally, apply Lemma \ref{MR3.4}(1) with $T= \{\mathfrak{p}_1, \mathfrak{p}_2\}$, $F=K_3(\sqrt{\alpha})$ as above, \eqref{eq:dim1prime} and \eqref{eq:locp2dim}, we get
\begin{align*}
\dim_{\FF}\Sel(E^F/K_3) &= \dim_{\FF}\Sel(E/K_3) - 2 \dim_{\FF} V_T + \sum_{\mathfrak r \in T} \dim_{\FF} H^1_f((K_3)_\mathfrak{r}, E[2]) \\
&=  \dim_{\FF}\Sel(E/K_3) - 2\dim_{\mathbb{F}_2}\mathrm{loc}_{\{\mathfrak p_1\}}(  \mathrm{Sel}_2(E/K_3) ) - 2 \dim_{\mathbb{F}_2}\mathrm{loc}_{\{\mathfrak{p}_2\}}(  \mathrm{Sel}_2(E/K_3) )  \\
&\quad \quad +  \dim_{\FF} H^1_f((K_3)_{\mathfrak{p}_1}, E[2]) + \dim_{\FF} H^1_f((K_3)_{\mathfrak{p}_2}, E[2]) \\
&= \dim_{\FF}\Sel(E/K_3) - 2
\end{align*}
Noting again the twist by $F$ above is a square norm twist, we have the desired result. 
\end{proof}

\section{Proofs of the Main Theorems}

We are now ready to prove Theorem \ref{thm:mainsel}. We'll then show how it implies Theorem \ref{thm:mainrank}. Again, the ``infinitely many'' of both theorems is quantified by Proposition \ref{prop:manytwists}.

\begin{proof}[Proof of Theorem \ref{thm:mainsel}]
Let $E$ and $K_3$ be as in Theorem \ref{thm:mainsel}. If $\dim_{\FF} \Sel (E/K_3) \equiv 0 \pmod{2}$, repeated application of Proposition \ref{lem:selrank-2} gives a square norm twist $L$ such that \newline $\dim_{\FF} \Sel(E^L/K_3)= 0$. Once we have $L$, Proposition \ref{prop:manytwists} provides infinitely many more square norm twists $F$ with $\dim_{\FF} \Sel(E^F/K_3)= 0$. 

Likewise, if $\dim_{\FF} \Sel (E/K_3) \equiv 1 \pmod{2}$, the argument above provides infinitely many more square norm twists $F$ with $\dim_{\FF} \Sel(E^F/K_3)= 1$. 
\end{proof}

Now we can prove  Theorem \ref{thm:mainrank} as a consequence of Theorem \ref{thm:mainsel}.

\begin{proof}[Proof of Theorem \ref{thm:mainrank}] 
Let $E$ and $K_3$ be as in the statement of the theorem. Theorem \ref{thm:mainrank} is essentially an immediate consequence of Theorem \ref{thm:mainsel} coupled with the upper bound the dimension of the 2-Selmer group provides for the rank. 

As in Theorem \ref{thm:mainsel}, there are two cases. In the first case, $\dim_{\mathbb{F}_2}\mathrm{Sel}_2(E/K_3)$ is even. In this case, Theorem \ref{thm:mainsel} provides infinitely many square norm twists $E^F/K_3$ for which $\dim_{\mathbb{F}_2}\mathrm{Sel}_2(E^F/K_3) = 0$. 

From \eqref{eq:rkbound}, we have $\mathrm{rk}(E/K_3) \leq \dim_{\FF}\Sel(E/K_3)$. Thus, our infinitely many square norm twists $E^F$  of 2-Selmer rank zero give us 
$$0= \dim_{\FF}\Sel(E^F/K_3) \geq \mathrm{rk}(E^F/K_3) = \mathrm{rk}(E/F) - \mathrm{rk}(E/K_3).$$
If $K/\mathbb{\Q}$ is the $S_4$-quartic corresponding to quadratic extension $F/K_3$ corresponding to each square norm twist, then having no rank growth from $K_3$ to $F$ means we have no rank growth from $\Q$ to $K$ for infinitely many $K$.

In the second case, $\dim_{\mathbb{F}_2}\mathrm{Sel}_2(E/K_3)$ is even. Then, in the same way as above, there are infinitely many  square norm twists $E^F$ of 2-Selmer rank one. The result follows if we assume the parity of the rank and 2-Selmer dimension are the same. 

\end{proof}

\section*{Acknowledgements}
Much of this project was completed as part of the author's graduate work; he thanks his advisor, Robert Lemke Oliver, for all of his guidance, suggestions, and support. He also thanks George McNinch, Sun Woo Park, Ari Shnidman, David Smyth, and Jiuya Wang for fruitful conversations and helpful comments on this paper. Finally, the author also extends his thanks to an anonymous referee from PJM for multiple careful readings; their comments greatly improved the paper.

\end{document}